\documentclass[12pt]{amsart}
\usepackage{amssymb,amsbsy}
\usepackage{amscd}
\usepackage{enumerate}
\usepackage{setspace}
\usepackage{tikz}
\usepackage{bm}
\usepackage{cite}
\usepackage{graphicx}
\usepackage{hyperref}
\usepackage{nicematrix}
\usepackage{colortbl}
\allowdisplaybreaks

\usepackage{kpfonts}
\usepackage{stackengine}
\usepackage{calc}
\newlength\shlength
\newcommand\xshlongvec[2][0]{\setlength\shlength{#1pt}%
  \stackengine{-5.6pt}{$#2$}{\smash{$\kern\shlength%
    \stackengine{7.55pt}{$\mathchar"017E$}%
      {\rule{\widthof{$#2$}}{.57pt}\kern.4pt}{O}{r}{F}{F}{L}\kern-\shlength$}}%
      {O}{c}{F}{T}{S}}
  
\usepackage[letterpaper]{geometry}
\theoremstyle{definition}
\newtheorem{theorem}{Theorem}[section]
\newtheorem{lemma}[theorem]{Lemma}
\newtheorem{proposition}[theorem]{Proposition}
\newtheorem*{proposition*}{Proposition}

\newtheorem{corollary}[theorem]{Corollary}

\newtheorem{example}[theorem]{Example}

\newcommand{\ovl}[1]{\overline{#1}}	
\newcommand{\parn}[1]{\left( #1 \right)}
\newcommand{\conv}[1]{\mathrm{conv}\parn{#1}}	

\newcommand{\Z}{{\mathbb{Z}}}
\newcommand{\R}{{\mathbb{R}}}
\newcommand{\C}{{\mathbb{C}}}
\newcommand{\N}{{\mathbb{N}}}

\renewcommand{\Re}{\ensuremath{\mathrm{Re}}}
\renewcommand{\Im}{\ensuremath{\mathrm{Im}}}

\newcommand\ce{r}

\title[Numerical range of a tridiagonal operator]{The numerical range of some  periodic tridiagonal operators is the convex hull of the numerical ranges of two finite matrices}

\date{January 2021}

\author{Benjam\'in A.~Itz\'a-Ortiz}
\author{Rub\'en A.~Mart\'inez-Avenda\~no}
\author{Hiroshi Nakazato}

\address{Centro de Investigaci\'on en Matem\'aticas, Universidad Aut\'onoma del Estado de Hidalgo, Pachuca, Hidalgo, Mexico}
\address{Departamento Acad\'emico de Matem\'aticas, Instituto Tecnol\'ogico Aut\'onomo de M\'exico, Mexico City, Mexico}
\address{Department of Mathematics and Physics, Hirosaki University, Hirosaki City, Japan}

\thanks{The second author's research is partially supported by the Asociaci\'on Mexicana de Cultura A.C.}

\begin{document}
\setlength\arraycolsep{2pt}

\begin{abstract}
In this paper we  prove a conjecture stated by the first two authors establishing the closure of the numerical range of a certain class of $n+1$-periodic tridiagonal operators as the convex hull of the numerical ranges of two tridiagonal $(n+1) \times (n+1)$ matrices. Furthermore, when $n+1$ is odd, we show that the  size of such matrices simplifies to $\frac{n}{2}+1$.
\end{abstract}

\maketitle

\section*{Introduction}

Tridiagonal operators and matrices have long been of interest to researchers  for a  wide variety of reasons, such as solving linear systems of equations, applications of finite differences for the numerical solution of differential equations, and finding roots of polynomials, just to name a few. 
They also have found applications in Physics. For example, a tridiagonal operator is used as the
  ``hopping sign model'' introduced in \cite{Feinberg} and  studied by many other authors, such as in \cite{BebianoEtAl,CWChLi10,
CWChLi13,ChD,HaggerJFA, Hagger,HI-O2016}. 
Though it may seem natural to try finding the numerical range of arbitrary tridiagonal operators and matrices, this turns out to be a hard problem. However, under specific additional conditions, some progress has been made \cite{CS,HI-O2016,IM,IMN}.
In \cite{BebianoEtAl}, Bebiano et al.~introduced the symbol matrix of a biperiodic banded operator and 
expressed the closure of the numerical range of such an operator as the convex hull of the union of the numerical ranges of an infinite number of symbol matrices.
In \cite{IM}, the first two authors of this paper defined the symbol matrix of an $(n+1)$-periodic tridiagonal operator and showed that the closure of the numerical range of such an operator is also the convex hull of the union of the numerical ranges of an infinite number of symbol matrices. They also provided an expression for the closure of the numerical range of a specific tridiagonal operator as the convex hull of the numerical ranges of two concrete matrices. In \cite{IMN} it is shown that the
 closure of the numerical range of an $(n+1)$-periodic tridiagonal operator may be expressed as the numerical range of a single but abstract $2(n+1)\times
 2(n+1)$ matrix. In this paper, we find explicit expressions for the matrix above when the tridiagonal operator satisfies certain conditions.
 
 Specifically, for a fixed natural number $n$, consider the $(n+1)$-periodic sequence of real numbers $a=\cdots a_{-1} a_0 a_1 a_2 \cdots$; that is, $a_k=a_{k+n+1}$ for every integer $k$. We define $A_a$ to be the bounded operator with infinite matrix
\begin{equation}\label{eq:Ab}
 A_{a}= 
\begin{pmatrix}
\ddots & \ddots & & & & & \\
\ddots & 0 & 1 & & & & \\
& a_{-2} & 0 & 1 & & & \\
& & a_{-1} & \framebox[0.4cm][l]{0} & 1 & & \\
& & & a_{0} & 0 & 1 & \\
& & & & a_{1} & 0 & \ddots \\
& & & & & \ddots & \ddots
\end{pmatrix}.
\end{equation}
We  prove that when $a_1=1$ and the finite sequence $a_2a_3\cdots a_na_0$ is palindromic, that is $a_2a_3\cdots a_na_0=a_0a_n \cdots a_3 a_2$, then
 the closure of the numerical range of $A_a$ is the convex hull of two $(n+1)\times (n+1)$ tridiagonal matrices. This, in particular, proves 
Conjecture~3.7 in \cite{IM}.  

We divide this work in three  sections. In Section~1 we briefly introduce the notation and terminologies needed in the rest of the paper. In Section~2
 the confirmation of Conjecture~3.7 in \cite{IM}
 is presented, among other examples. And finally in Section~3, when $n+1$ is odd, a further simplification of the expression for the closure of numerical range of $T$ obtained in Section~2 is derived.

\section{Preliminaries}

In this section we present the notation and basic facts that we will use throughout this paper. The symbols $\N$, $\N_0$, $\Z$, $\R$ and $\C$ denote the set of positive integers, the sets of nonnegative integers, the set of integers, the set of real numbers and the set of complex numbers, respectively. Also, recall that if $\mathcal A$ is a nonempty set, ${\mathcal A}^{\N_0}$ denotes the set of sequences indexed by $\N_0$ with values on $\mathcal A$.

For a fixed $n\in\N$, let $a$, $b$ and $c$ be $(n+1)$-periodic infinite sequences in $\mathcal A^{\N_0}$. The  $(n+1)$-periodic tridiagonal operator on $\ell^2(\N_0)$ given by the infinite matrix
\[
\begin{pmatrix}
b_0 & c_0 & & & & & & & &\\
a_1 & b_1 & c_1 & & & & & & & \\
    & a_{2} & b_2 & c_2 & & & & & & \\
    &       & \ddots & \ddots  & \ddots & & & & & \\
    & & & a_{n} & b_n & c_n & & & &\\
    & & & & a_{0} & b_0 & c_0 & & &\\
    & & & & & \ddots & \ddots & \ddots & &\\
    & & & & & &a_{n-1} & b_{n-1}& c_{n-1} & \\
    & & & & & & &  a_n& b_n  & c_n &\\
    & & & & & & & & \ddots &  \ddots&\ddots
\end{pmatrix},
\]
will be denoted by $T=T(a,b,c)$.

Since the sum of the moduli of the entries in each column and each row is uniformly bounded (see, e.g., \cite[Example 2.3]{Kato}), the operator $T$ is bounded. Also, observe that the operator given by the biinfinite matrix $A_a$ is also a bounded operator, if the biinfinite sequence $a$ has finitely many values.

The numerical range of a bounded operator $A$ on a Hilbert space $\mathcal H$, with inner product denoted by $\left< \cdot , \cdot \right>$, is the set of complex numbers given by 
\[
W(A)= \{ \left< A x, x \right> \, : \, \| x \|=1 \}.
\]
This set is a bounded convex subset of $\C$, by the Toeplitz--Hausdorff Theorem, and it is closed if the Hilbert space is finite dimensional. It follows from the convexity that the closure of the numerical range can be seen as the intersection of the closed half-planes containing $W(A)$.

In \cite{Kip} (see also \cite{ZH}), R.~Kippenhahn showed that, for a square matrix $A$, the vertical lines $\Re(z)=\lambda_1(A)$ and $\Re(z)=\lambda_n(A)$ are support lines of $W(A)$. Here $\lambda_1(A)$ and $\lambda_n(A)$ are the largest and smallest eigenvalues of $\Re(A)$, respectively. (Recall that $\Re(A):=\frac{1}{2} (A + A^*)$ and $\Im(A):=\frac{1}{2i} (A - A^*)$.) In fact, he showed that if $\alpha \in W(A)$ then  $\lambda_n(A) \leq \Re(\alpha) \leq \lambda_1(A)$ and the equalities hold for some points $\alpha_1, \alpha_2 \in W(A)$). 

It is clear that for each $\theta \in [0, 2\pi)$ we have $e^{i \theta} W(A)=W(e^{i \theta} A)$. Now, if $\alpha \in W(A)$, then $e^{-i\theta} \alpha \in W(e^{-i\theta} A)$ and hence $\Re(e^{-i\theta} \alpha) \leq \lambda_1(e^{-i \theta} A)$. Therefore the lines $\Re(e^{-i\theta} z) = \lambda_1(e^{-i \theta} A)$ are support lines of $W(A)$. It follows that $W(A)$, being a convex set, is uniquely determined by the real numbers $\lambda_1(e^{-i\theta} A)$, as $\theta$ varies on the interval $[0, 2\pi)$. In other words, the largest eigenvalue of $\Re(e^{-i \theta} A)$, which equals $\cos(\theta) \Re(A) + \sin(\theta) \Im(A)$, determines the set $W(A)$. Therefore, the largest roots of the family of characteristic polynomials
\[
\det(t I_n - \cos(\theta) \Re(A) - \sin(\theta) \Im(A)),
\]
completely characterizes the convex set $W(A)$.

We can then define the Kippenhahn polynomial of $A$ as the homogeneous polynomial $F_A(t,x,y)=\det(t I_n + x \Re(A) + y \Im(A))$. It is then clear that two matrices have the same numerical range if their Kippenhahn polynomials coincide.  In fact,
\[
\max\{ t \in \R \, : \, F_A(t,-\cos(\theta),-\sin(\theta))=0 \} = \max\{ \Re(e^{-i \theta} z) \, : \, z \in W(A)\}
\]
for each $\theta \in [0,2\pi)$.

For the rest of this paper we will use the following notation. For $0 \leq j < n$ we define
 \begin{equation*}
     \alpha_j=\frac{c_j+\overline{a_{j+1}}}{2},\quad \gamma_j=\frac{c_j-\overline{a_{j+1}}}{2i},
     \end{equation*}
     and
\begin{equation*}     
     \alpha_n=\frac{a_0+\overline{c_n}}{2},\quad
     \gamma_n=\frac{a_0-\overline{c_n}}{2i}.
 \end{equation*}

The following Proposition, which was shown in \cite{IMN}, will be useful in what follows.

\begin{proposition}\label{prop:defP}
Let $n \in \N$. Suppose that $T(a, b, c)$ is an $n+1$-periodic tridiagonal operator acting on $\ell^2({\mathbb N}_0)$. Let $P$ be the real homogeneous polynomial of degree $2(n+1)$ given by
\begin{equation*}
   P(t, x, y)=\left( G_{n}(t, x, y)-|\alpha_n x +\gamma_n y|^2 H_{n}(t, x, y)\right)^2 
   -4 \prod_{j=0}^{n} \left|\alpha_j x +\gamma_j y\right|^2.
\end{equation*}
   Then
  \[
  \sup \left\{\Re(e^{-i \theta} z)\colon\right. \bigl. z \in W\left(T(a, b,c)\right)\bigr\} 
   =\max \{ t \in {\mathbb R}\colon P(t, -\cos \theta, -\sin \theta) =0 \},
  \]
  for each $\theta \in [0, 2\pi)$.

   Here $G_n(t,x,y)$ is the determinant of the tridiagonal $(n+1) \times (n+1)$ matrix
     \[
\begin{pmatrix}
    \lambda_{1,1} & \lambda_{1,2} & 0  & 0 &  \ldots &  0 &0 \\
    \lambda_{2,1} & \lambda_{2,2} & \lambda_{2,3} & 0  & \cdots  &0 &  0  \\
    0 & \lambda_{3,2} & \lambda_{3,3} & \lambda_{3,3}  & \cdots  & 0 &  0  \\
    0 & 0 & \lambda_{4,3} & \lambda_{4,4}  & \cdots  &  0  & 0\\
       \vdots & \vdots & \vdots & \vdots & &\vdots &\vdots  \\
       0  &  0  &   0 &  0 & \cdots&    \lambda_{n, n} & \lambda_{n, n+1} \\
       0  &  0  &   0 &  0 & \cdots &    \lambda_{n+1, n} & \lambda_{n+1, n+1} 
\end{pmatrix},
\]
and, where we set $H_n(t,x,y)=1$ when $n=1$, and, for $n\geq 2$, we set $H_n(t,x,y)$ to be the determinant of $(n-1) \times (n-1)$ tridiagonal matrix  
  \[
\begin{pmatrix}
      \lambda_{2,2} & \lambda_{2,3} & 0  &  \cdots & 0 & 0  \\
      \lambda_{3,2} & \lambda_{3,3} & \lambda_{3,4} &  \cdots  & 0 &  0  \\
      0  & \lambda_{4,3} & \lambda_{4,4} & \cdots  &  0 & 0\\
       \vdots & \vdots & \vdots & & \vdots & \vdots \\ 
       0  &  0 & 0  &  \cdots &    \lambda_{n-1, n-1} & \lambda_{n-1, n} \\
              0 & 0 &  0  &   \cdots &    \lambda_{n, n-1} & \lambda_{n, n} 
\end{pmatrix},
\]
and we have set, for $1 \leq j \leq n+1$,
\[
\lambda_{j,j}=t +\Re(b_{j-1}) x +\Im(b_{j-1}) y,
\]
and for $ 1 \leq j \leq n$, 
\[
\lambda_{j,j+1}=\alpha_{j-1} x + \gamma_{j-i} y \quad \text{ and} \quad \lambda_{j+1,j}=\ovl{\alpha_{j-1}} x + \ovl{\gamma_{j-1}} y.
\]
\end{proposition}

We will need the following lemma.

\begin{lemma}\label{le:laplace}
Consider the $(n+1)\times(n+1)$ ``almost tridiagonal'' matrix
\[
U=\begin{pmatrix} 
u_{1,1} & u_{1,2} & 0 &  0 & \dots & 0 & 0 & u_{1,n+1} \\
u_{2,1} & u_{2,2} & u_{2,3} & 0 &  \dots & 0 & 0 & 0\\
0 & u_{3,2} & u_{3,3} & u_{3,4} & \dots & 0 & 0 & 0  \\
0 & 0 & u_{4,3} & u_{4,4} & \dots & 0 & 0 & 0 \\
\vdots & \vdots & \vdots & \vdots &\ddots &\vdots &\vdots &\vdots \\
0 & 0 & 0 & 0 & \dots & u_{n-1,n-1} & u_{n-1,n} & 0  \\
0 & 0 & 0 & 0 & \dots & u_{n,n-1} & u_{n,n} & u_{n,n+1}  \\
u_{n+1,1} & 0 & 0 & 0 & \dots & 0 & u_{n+1,n} & u_{n+1,n+1}  
\end{pmatrix},
\]
where every $u_{i,j} \in \C$. Then, $\det(U)$ equals
\begin{equation*}
    \begin{split}
&
\det\begin{pmatrix}
u_{1,1} & u_{1,2} & 0 &  \dots & 0 & 0 \\
u_{2,1} & u_{2,2} & u_{2,3} &  \dots & 0 & 0\\
0 & u_{3,2} & u_{3,3} &  \dots & 0 & 0  \\
\vdots & \vdots & \vdots &\ddots &\vdots &\vdots \\
0 & 0 & 0 & \dots & u_{n-1,n} & 0  \\
0 & 0 & 0 & \dots & u_{n,n} & u_{n,n+1}  \\
0 & 0 & 0 & \dots & u_{n+1,n} & u_{n+1,n+1}
\end{pmatrix} 
- u_{1,n+1} u_{n+1,1} 
\det \begin{pmatrix}
u_{2,2} & u_{2,3} & \dots & 0 & 0 \\
u_{3,2} & u_{3,3} &  \dots & 0 & 0 \\
\vdots & \vdots &\ddots &\vdots &\vdots  \\
0 & 0 & \dots & u_{n-1,n-1} & u_{n-1,n}  \\
0 & 0 & \dots & u_{n,n-1} & u_{n,n}  
\end{pmatrix} \\
\\
&+ (-1)^{n} u_{n+1,1} u_{1,2} u_{2,3} \cdots u_{n-1,n} u_{n,n+1} 
+ (-1)^{n}u_{1,n+1} u_{2,1} u_{3,2} \cdots u_{n,n-1} u_{n+1,n}.
\end{split}
\end{equation*}
\end{lemma}
\begin{proof}
This lemma can be proved by succesive applications of the multilinearity of the determinant function and the Laplace Expansion Theorem. We leave the long, but straightforward proof, to the reader.
\end{proof}

\section{Proof of the conjecture} 

 We shall prove that Conjecture 3.7 of \cite{IM} is true by using the polynomial given in Proposition~\ref{prop:defP}. We first state some lemmas which will be useful.
  
    \begin{lemma}\label{le:Apm}
Let $n \in \N$, $n \geq 2$, let $\ce_0, \ce_1, \ce_2, \dots \ce_n \in \C$ and consider the two $(n+1)\times (n+1)$ matrices
\begin{equation}\label{eq:Apm}
\begin{pmatrix}
t & \ce_0 & 0 & 0 & 0 & 0 & \cdots &0  & 0 & \pm \ce_n \\
\ce_0 & t & \ce_1 & 0 & 0 & 0 & \cdots &0  & 0 & 0 \\
0 & \ce_1 & t & \ce_2 & 0  & 0 & \cdots &0  & 0 & 0 \\
0 & 0 & \ce_2 & t & \ce_3 & 0 & \cdots &0  & 0 & 0 \\
0 & 0 & 0 & \ce_3 & t & \ce_4 & \cdots &0  & 0 & 0 \\
0 & 0 & 0 & 0 & \ce_4 & t & \cdots &0  & 0 & 0 \\
\vdots & \vdots & \vdots & \vdots & \vdots & \vdots & \ddots & \vdots & \vdots & \vdots \\
0 & 0 & 0 & 0 & 0 & 0  & \cdots & t & \ce_{n-2} & 0 \\
0 & 0 & 0 & 0 & 0 & 0  & \cdots & \ce_{n-2 }& t & \ce_{n-1} \\
\pm\ce_n & 0 & 0 & 0 & 0 & 0  & \cdots & 0 & \ce_{n-1} & t \end{pmatrix}.
\end{equation}
\begin{itemize} 
\item If $n+1$ is odd, assume that $\ce_j=\ce_{n-j+1}$ for $j=1, 2, \dots, \frac{n}{2}$ (we make no assumption on $\ce_0$).
Then the determinant of each of the matrices in \eqref{eq:Apm} equals the product
\begin{equation*}
\det\begin{pmatrix}
t\pm \ce_0 & \ce_1 & 0 & \dots & 0 & 0 & 0 \\
\ce_1 & t & \ce_2 & \dots & 0 & 0 & 0 &  \\
0 & \ce_2 & t & \dots  & 0  & 0 & 0  \\
\vdots & \vdots & \vdots & \ddots & \vdots & \vdots & \vdots \\
0 & 0 & 0 & \dots & t & \ce_{\frac{n}{2}-1} & 0 & \\
0 & 0 & 0 & \dots & \ce_{\frac{n}{2}-1} & t  &   \ce_{\frac{n}{2}} \\
0 & 0 & 0 & \dots & 0 & 2 \ce_{\frac{n}{2}} & t \\
\end{pmatrix}
\det\begin{pmatrix}
t\mp \ce_0 &\ce_{1} & 0 & \cdots & 0 & 0 \\
\ce_{1} & t & \ce_{2} & \cdots  & 0 & 0 \\
0 &  \ce_{2} & t & \cdots  & 0 & 0 \\
\vdots & \vdots & \vdots & \ddots & \vdots & \vdots \\
0 & 0 & 0 & \cdots & t & \ce_{\frac{n}{2}-1} \\
0 & 0 & 0 & \cdots & \ce_{\frac{n}{2}-1} & t 
\end{pmatrix}.
\end{equation*}
Observe that, if $n+1=3$, the matrix in the determinant in the right-hand-side of the expression above is the $1\times 1$ matrix $(t \mp \ce_0)$. 

\item If $n+1$ is even, assume that $\ce_j=\ce_{n-j+1}$ for $j=1, 2, \dots, \frac{n-1}{2}$ (we make no assumption on $\ce_0$, nor on $\ce_{\frac{n+1}{2}}$). 
Then the determinant of each of the matrices in \eqref{eq:Apm} equals the product
\begin{equation*}
    \det\begin{pmatrix}
t\pm\ce_0 & \ce_1 & 0 & \dots & 0 & 0 & 0 \\
\ce_1 & t & \ce_2 & \dots & 0 & 0 & 0 &  \\
0 & \ce_2
& t & \dots  & 0  & 0 & 0  \\
\vdots & \vdots & \vdots & \ddots & \vdots & \vdots & \vdots \\
0 & 0 & 0 & \dots & t & \ce_{\frac{n+1}{2}-2} & 0 & \\
0 & 0 & 0 & \dots & \ce_{\frac{n+1}{2}-2} & t  &   \ce_{\frac{n+1}{2}-1} \\
0 & 0 & 0 & \dots & 0 &  \ce_{\frac{n+1}{2}-1} & t+\ce_{\frac{n+1}{2}} \\
\end{pmatrix}
\det\begin{pmatrix}
t\mp\ce_0 & \ce_{1} & 0 & \dots & 0 & 0 & 0 \\
\ce_{1} & t & \ce_2 & \dots & 0 & 0 & 0 &  \\
0 & \ce_2
& t & \dots  & 0  & 0 & 0  \\
\vdots & \vdots & \vdots & \ddots & \vdots & \vdots & \vdots \\
0 & 0 & 0 & \dots & t & \ce_{\frac{n+1}{2}-2} & 0 & \\
0 & 0 & 0 & \dots & \ce_{\frac{n+1}{2}-2} & t  &   \ce_{\frac{n+1}{2}-1} \\
0 & 0 & 0 & \dots & 0 &  \ce_{\frac{n+1}{2}} & t-\ce_{\frac{n+1}{2}}
\end{pmatrix}.
\end{equation*}
\end{itemize}
\end{lemma}
\begin{proof}
Let $Q_{\pm}$ be the matrices
\[
\begin{pmatrix}
0 & 1 & 0 & \dots & 0 & 0 \\
0 & 0 & 1 &  \dots & 0 & 0 \\
0 & 0 & 0 &  \dots & 0 & 0\\
\vdots & \vdots & \vdots & \ddots & \vdots & \vdots \\
0 & 0 & 0  & \dots & 0 & 1 \\
\pm 1 & 0 & 0 &  \dots & 0 & 0\\
\end{pmatrix}.
\]
For each of the matrices in \eqref{eq:Apm}, we left-multiply by $Q_{+}$ (respectively, $Q_{-}$) and right-multiply by $Q_{+}^{-1}$ (respectively, $Q_{-}^{-1}$) to obtain two matrices which are centrosymmetric (see, e.g., \cite{Aitken}, for the definition). It is known that the determinant of a centrosymmetryc matrix can be written as the product of the determinants of two smaller matrices (see, \cite[p.~125, eq. (9)]{Aitken}) from which the result follows. (See \cite{IMN_lemmas} for the full details of the proof.) 
\end{proof}

\begin{lemma}\label{le:Bpm}
Let $n \in \N$, $n \geq 2$, let $\ce_0, \ce_1, \ce_2, \dots \ce_n \in \C$ and consider the two $(n+1)\times (n+1)$ matrices
\begin{equation}\label{eq:Bpm}
\begin{pmatrix}
t\pm \ce_0 & \ce_1 & 0 & 0 & 0 & 0 & \cdots &0  & 0 & 0 \\
\ce_1 & t & \ce_2 & 0 & 0 & 0 & \cdots &0  & 0 & 0 \\
0 & \ce_2 & t & \ce_3 & 0  & 0 & \cdots &0  & 0 & 0 \\
0 & 0 & \ce_3 & t & \ce_4 & 0 & \cdots &0  & 0 & 0 \\
0 & 0 & 0 & \ce_4 & t & \ce_5 & \cdots &0  & 0 & 0 \\
0 & 0 & 0 & 0 & \ce_5 & t & \cdots &0  & 0 & 0 \\
\vdots & \vdots & \vdots & \vdots & \vdots & \vdots & \ddots & \vdots & \vdots & \vdots \\
0 & 0 & 0 & 0 & 0 & 0  & \cdots & t & \ce_{n-1} & 0 \\
0 & 0 & 0 & 0 & 0 & 0  & \cdots & \ce_{n-1 }& t & \ce_{n} \\
0 & 0 & 0 & 0 & 0 & 0  & \cdots & 0 & \ce_{n} & t\pm \ce_0 
\end{pmatrix}.
\end{equation}
\begin{itemize}
    \item If $n+1$ is odd, assume that $\ce_j=\ce_{n-j+1}$ for $j=1, 2, \dots, \frac{n}{2}$ (we make no assumption on $\ce_0$). Then the determinant of each of the matrices in \eqref{eq:Bpm} equals the product
\begin{equation*}
    \det\begin{pmatrix}
    t\pm\ce_0 & \ce_1 & 0 & \cdots & 0 & 0 \\
    \ce_1 & t & \ce_2 & \cdots & 0 & 0 \\
    0 & \ce_2 & t & \cdots & 0 & 0  \\
    \vdots & \vdots & \vdots & \ddots  &\vdots & \vdots \\
    0 & 0 & 0 & \cdots & t & \ce_{\frac{n}{2}-1} \\
    0 & 0 & 0 & \cdots &  \ce_{\frac{n}{2}-1} & t\\
    \end{pmatrix}
    \det\begin{pmatrix}
    t\pm\ce_0 & \ce_1 & 0 & \cdots & 0 & 0 & 0\\
    \ce_1 & t & \ce_2 & \cdots & 0 & 0 & 0\\
    0 & \ce_2 & t & \cdots & 0 & 0 & 0 \\
    \vdots & \vdots & \vdots & \ddots  &\vdots & \vdots & \vdots\\
    0 & 0 & 0 & \cdots & t & \ce_{\frac{n}{2}-1} & 0\\
    0 & 0 & 0 & \cdots &  \ce_{\frac{n}{2}-1} & t & \ce_\frac{n}{2}\\
    0 & 0 & 0 & \cdots & 0 & 2 \ce_{\frac{n}{2}} & t\\
    \end{pmatrix}.
\end{equation*}
Observe that, if $n+1=3$, the matrix in the determinant in the left-hand-side of the expression above is the $1\times 1$ matrix $(t \pm \ce_0)$.

\item If $n+1$ is even, assume that $\ce_j=\ce_{n-j+1}$ for $j=1, 2, \dots, \frac{n-1}{2}$ (we make no assumption on $\ce_0$, nor on $\ce_{\frac{n+1}{2}}$). Then the determinant of each of the matrices in \eqref{eq:Bpm} equals the product
\[   
\det\begin{pmatrix}
    t\pm\ce_0 & \ce_1 & 0 & \cdots & 0 & 0 \\
    \ce_1 & t & \ce_2 & \cdots & 0 & 0 \\
    0 & \ce_2 & t & \cdots & 0 & 0  \\
    \vdots & \vdots & \vdots & \ddots  &\vdots & \vdots \\
    0 & 0 & 0 & \cdots & t & \ce_{\frac{n-1}{2}} \\
    0 & 0 & 0 & \cdots &  \ce_{\frac{n-1}{2}} & t+\ce_{\frac{n+1}{2}}\\
    \end{pmatrix}
        \det\begin{pmatrix}
    t\pm\ce_0 & \ce_1 & 0 & \cdots & 0 & 0 \\
    \ce_1 & t & \ce_2 & \cdots & 0 & 0 \\
    0 & \ce_2 & t & \cdots & 0 & 0  \\
    \vdots & \vdots & \vdots & \ddots  &\vdots & \vdots \\
    0 & 0 & 0 & \cdots & t & \ce_{\frac{n-1}{2}} \\
    0 & 0 & 0 & \cdots &  \ce_{\frac{n-1}{2}} & t- \ce_{\frac{n+1}{2}}\\
    \end{pmatrix}.
    \]
    \end{itemize}
\end{lemma}
\begin{proof}
The matrices in \eqref{eq:Bpm} are centrosymmetric. Again, applying the factorization given in \cite[p.~125, eq. (9)]{Aitken}, the result follows. (See \cite{IMN_lemmas} for the full details of the proof.) 
\end{proof}

Now we prove the main theorem of this section.  We use the notation $\conv{A \cup B}$ to denote the convex hull of the sets $A$ and $B$.

\begin{theorem}\label{main}
For $n \in \N$, let 
$a$ and  $c$  be $n+1$-periodic infinite sequences in $\mathcal \R^{\N_0}$. Let
$a_0 a_1 a_2 \cdots a_n$  and $c_0 c_1 c_2 \cdots c_n$ be the period words of $a$ and $c$, respectively.
Suppose that 
\[c_0=a_1, \quad |c_1+a_2| =|c_n+a_0|, \quad |c_1-a_2|=|c_n-a_0|,\]
\begin{equation*}
|c_j+a_{j+1}| = |c_{n-j+1}+a_{n-j+2}|, \quad \text{ and } \quad |c_j-a_{j+1}|  =|c_{n-j+1}-a_{n-j+2}|,
\end{equation*} 
for all $j=2, 3, \dots \left[\frac{n}{2}\right]$.

Then 
\[
    \ovl{W(T(a, 0,c))}=\conv{W(A\sp{+})\cup W(A\sp{-}}),   \]
where $A^{\pm}$ are the $(n+1)\times (n+1)$ matrices
\[
A^{\pm}:=\begin{pmatrix}
\pm a_1 & c_1 & 0 & 0 & \cdots & 0 & 0 & 0\\
a_2 & 0 & c_2 & 0 & \cdots & 0 & 0 & 0\\
0 & a_3 & 0 & c_3 & \cdots & 0 & 0 & 0\\
0 & 0 & a_4 & 0 & \cdots & 0 & 0 & 0\\
\vdots & \vdots & \vdots & \vdots &  \ddots &\vdots & \vdots & \vdots \\
0 & 0 & 0 & 0 & \cdots & 0 & c_{n-1} & 0 \\
0 & 0 & 0 & 0 & \cdots & a_n & 0 & c_n \\
0 & 0 & 0 & 0 & \cdots & 0 & a_0 & \pm c_0
\end{pmatrix}.
\]
\end{theorem} 
  
\begin{proof}
By Proposition~\ref{prop:defP},
  \[
  \sup \left\{\Re(e^{-i \theta} z)\colon\right. \bigl. z \in W\left(T(a, 0,c)\right)\bigr\} 
   =\max \{ t \in {\mathbb R}\colon P(t, -\cos \theta, -\sin \theta) =0 \},
  \]
  for each $\theta \in [0, 2\pi)$, where
  \begin{equation*}
P(t, x, y) = \left( G_n(t, x, y) - |\alpha_n x +\gamma_n y|^2 H_n(t, x, y)\right)^2 -4 { \prod_{j=0}^{n}} \left| \alpha_j x +\gamma_j y \right|^2,
\end{equation*}
where $G_n(t,x,y)$ and $H_n(t,x,y)$ are as in the mentioned proposition.

Let $A:=\begin{pmatrix} A^+ & 0 \\ 0 & A^-\end{pmatrix}$, then clearly $F_A(t,x,y)=F_{A^+}(t,x,y) \cdot F_{A^-}(t,x,y)$. We will show that $F_A(t,x,y)$ equals $P(t,x,y)$ and hence 
\[
\ovl{W(T(a,0,c))}=W(A).
\]
But since $A$ is block-diagonal, we also have
\[
W(A)=\conv{W(A^+) \cup W(A^-)},
\]
which will complete the proof.

To achieve this, we will find a suitable factorization of $P(t,x,y)$ and then we will observe that one of the factors is $F_{A^+}(t,x,y)$ and the other is $ F_{A^-}(t,x,y)$.

By the hypotheses, we have 
\begin{equation}\label{eq:hypothesis}
|\alpha_j|=|\alpha_{n-j+1}|, \quad \text{ and } \quad
|\gamma_j|=|\gamma_{n-j+1}|, 
\end{equation}
for $j=1, 2, 3, \dots, \left[ \frac{n}{2} \right]$. It is also clear that  $\ovl{\lambda_{j,j+1}}=\lambda_{j+1,j}$ for every $j$.

Since $x$ and $y$ are real variables and by hypothesis $c_0=a_1$, we have $\gamma_0=0$ and so $\lambda_{1,2}=\alpha_0x=\lambda_{2,1}$, it follows that
\begin{equation*}
G_n(t,x,y)=
    \det\begin{pmatrix}
    \lambda_{1,1} & \lambda_{1,2} & 0 & \dots & 0 & 0 & 0\\
    \lambda_{1,2} & \lambda_{2,2} & |\lambda_{2,3}| & \dots & 0 & 0 & 0\\
    0 & |\lambda_{2,3}| & \lambda_{3,3} & \dots & 0 & 0 & 0\\
    \vdots & \vdots & \vdots & \ddots & \vdots & \vdots & \vdots \\
    0 & 0 & 0 & \dots & \lambda_{n-1,n-1} & |\lambda_{n-1,n}| & 0 \\
    0 & 0 & 0 & \dots & |\lambda_{n-1,n}| & \lambda_{n,n} & |\lambda_{n,n+1}| \\    
    0 & 0 & 0 & \dots & 0 & |\lambda_{n,n+1}| & \lambda_{n+1,n+1}
    \end{pmatrix}
\end{equation*}
and, when $n+1\geq 3$,
\begin{equation*}
H_n(t,x,y)=
    \det\begin{pmatrix}
    \lambda_{2,2} & |\lambda_{2,3}| & \dots & 0 & 0 \\
    |\lambda_{2,3}| & \lambda_{3,3} & \dots & 0 & 0 \\
    \vdots & \vdots & \ddots & \vdots & \vdots \\
    0 & 0 & \dots & \lambda_{n-1,n-1} & |\lambda_{n-1,n}| \\
    0 & 0 & \dots & |\lambda_{n-1,n}| & \lambda_{n,n}
    \end{pmatrix}.
\end{equation*}

In order to simplify notation, let us set $\ce_0=\lambda_{1,2}=\alpha_0 x$, set $\ce_j=|\lambda_{j+1,j+2}|=|\alpha_j x + \gamma_j y|$ for $j=1, 2, \dots, n-1$, and set $\ce_n=|\alpha_n x + \gamma_n y|$. 

Observe that, since $\alpha_j$ is real and $\gamma_j$ is purely imaginary, then it follows by equation \eqref{eq:hypothesis} that
\[
r_j^2 = |\alpha_j|^2 x^2 + |\gamma_j|^2 y ^2 = |\alpha_{n-j+1}|^2 x^2 + |\gamma_{n-j+1}|^2 y ^2 = r_{n-j+1}^2,
\]
and hence $r_j=r_{n-j+1}$ for $j=1, 2, 3, \dots, \left[\frac{n}{2} \right]$.

We then obtain
\begin{equation*}
G_n(t,x,y)=
    \det\begin{pmatrix}
    t & \ce_0 & 0 & \dots & 0 & 0 & 0\\
    \ce_0 & t & \ce_1 & \dots & 0 & 0 & 0\\
    0 & \ce_1 & t & \dots & 0 & 0 & 0\\
    \vdots & \vdots & \vdots & \ddots & \vdots & \vdots & \vdots \\
    0 & 0 & 0 & \dots & t  & \ce_{n-2} & 0 \\
    0 & 0 & 0 & \dots & \ce_{n-2} & t & \ce_{n-1} \\    
    0 & 0 & 0 & \dots & 0 & \ce_{n-1} & t
    \end{pmatrix}
\end{equation*}
and, when $n+1\geq 3$,
\begin{equation*}
H_n(t,x,y)=
    \det\begin{pmatrix}
    t & \ce_1 & \dots & 0 & 0 \\
    \ce_1 & t & \dots & 0 & 0 \\
    \vdots & \vdots & \ddots & \vdots & \vdots \\
    0 & 0 & \dots & t & \ce_{n-2}\\
    0 & 0 & \dots & \ce_{n-2}& t \\
    \end{pmatrix}.
\end{equation*}

We then have
\begin{align*}
P(t,x,y)&=\left(G_n(t,x,y)-\ce_n^2 H_n(t,x,y)\right)^2 - 4 \prod_{j=0}^n \ce_j^2 \\
&=\left(G_n(t,x,y)-\ce_n^2 H_n(t,x,y)  - 2 \prod_{j=0}^n \ce_j\right) \left(G_n(t,x,y)-\ce_n^2 H_n(t,x,y) + 2 \prod_{j=0}^n \ce_j\right).
\end{align*}

We assume, for the moment, that $n+1\geq 3$. Then $F_A(t,x,y)$ equals, by Lemma~\ref{le:laplace},
\begin{equation}
\det\begin{pmatrix} \label{eq:FA}
    t & \ce_0 & 0 & \dots & 0 & 0 & \ce_n \\
    \ce_0 & t & \ce_1 & \dots & 0 & 0 & 0\\
    0 & \ce_1 & t & \dots & 0 & 0 & 0\\
    \vdots & \vdots & \vdots & \ddots & \vdots & \vdots & \vdots \\
    0 & 0 & 0 & \dots & t  & \ce_{n-2} & 0 \\
    0 & 0 & 0 & \dots & \ce_{n-2} & t & \ce_{n-1} \\    
    \ce_n & 0 & 0 & \dots & 0 & \ce_{n-1} & t
    \end{pmatrix}
    \det\begin{pmatrix}
    t & \ce_0 & 0 & \dots & 0 & 0 & -\ce_n\\
    \ce_0 & t & \ce_1 & \dots & 0 & 0 & 0\\
    0 & \ce_1 & t & \dots & 0 & 0 & 0\\
    \vdots & \vdots & \vdots & \ddots & \vdots & \vdots & \vdots \\
    0 & 0 & 0 & \dots & t  & \ce_{n-2} & 0 \\
    0 & 0 & 0 & \dots & \ce_{n-2} & t & \ce_{n-1} \\    
    -\ce_n & 0 & 0 & \dots & 0 & \ce_{n-1} & t
    \end{pmatrix}.
\end{equation}
We need to consider two cases:

{\bf Case $n+1$ is even.} By Lemma \ref{le:Apm}, since $r_j=r_{n-j+1}$ for $j=1, 2, 3, \dots, \frac{n-1}{2}$, the above product \eqref{eq:FA} equals

\begin{equation*}
\begin{split}
\det\begin{pmatrix}
t + \ce_0 & \ce_1 & 0 & \dots & 0 & 0 & 0 \\
\ce_1 & t & \ce_2 & \dots & 0 & 0 & 0 &  \\
0 & \ce_2 & t & \dots  & 0  & 0 & 0  \\
\vdots & \vdots & \vdots & \ddots & \vdots & \vdots & \vdots \\
0 & 0 & 0 & \dots & t & \ce_{\frac{n+1}{2}-2} & 0 & \\
0 & 0 & 0 & \dots & \ce_{\frac{n+1}{2}-2} & t  &   \ce_{\frac{n+1}{2}-1} \\
0 & 0 & 0 & \dots & 0 &  \ce_{\frac{n+1}{2}-1} & t+ \ce_{\frac{n+1}{2}} \\
\end{pmatrix}
\det\begin{pmatrix}
t - \ce_0 & \ce_1 & 0 & \dots & 0 & 0 & 0 \\
\ce_1 & t & \ce_2 & \dots & 0 & 0 & 0 &  \\
0 & \ce_2 & t & \dots  & 0  & 0 & 0  \\
\vdots & \vdots & \vdots & \ddots & \vdots & \vdots & \vdots \\
0 & 0 & 0 & \dots & t & \ce_{\frac{n+1}{2}-2} & 0 & \\
0 & 0 & 0 & \dots & \ce_{\frac{n+1}{2}-2} & t  &   \ce_{\frac{n+1}{2}-1} \\
0 & 0 & 0 & \dots & 0 &  \ce_{\frac{n+1}{2}-1} & t- \ce_{\frac{n+1}{2}} \\
\end{pmatrix}
\\
\det\begin{pmatrix}
t - \ce_0 & \ce_1 & 0 & \dots & 0 & 0 & 0 \\
\ce_1 & t & \ce_2 & \dots & 0 & 0 & 0 &  \\
0 & \ce_2 & t & \dots  & 0  & 0 & 0  \\
\vdots & \vdots & \vdots & \ddots & \vdots & \vdots & \vdots \\
0 & 0 & 0 & \dots & t & \ce_{\frac{n+1}{2}-2} & 0 & \\
0 & 0 & 0 & \dots & \ce_{\frac{n+1}{2}-2} & t  &   \ce_{\frac{n+1}{2}-1} \\
0 & 0 & 0 & \dots & 0 &  \ce_{\frac{n+1}{2}-1} & t+ \ce_{\frac{n+1}{2}}
\end{pmatrix}
\det\begin{pmatrix}
t + \ce_0 & \ce_1 & 0 & \dots & 0 & 0 & 0 \\
\ce_1 & t & \ce_2 & \dots & 0 & 0 & 0 &  \\
0 & \ce_2 & t & \dots  & 0  & 0 & 0  \\
\vdots & \vdots & \vdots & \ddots & \vdots & \vdots & \vdots \\
0 & 0 & 0 & \dots & t & \ce_{\frac{n+1}{2}-2} & 0 & \\
0 & 0 & 0 & \dots & \ce_{\frac{n+1}{2}-2} & t  &   \ce_{\frac{n+1}{2}-1} \\
0 & 0 & 0 & \dots & 0 &  \ce_{\frac{n+1}{2}-1} & t- \ce_{\frac{n+1}{2}}
\end{pmatrix}.
\end{split}
\end{equation*}
By applying Lemma \ref{le:Bpm} to the above equality, we conclude that $P(t,x,y)$ equals
\begin{equation*}
\det\begin{pmatrix}
t + \ce_0 & \ce_1 & 0 & 0 & 0 & 0 & \cdots &0  & 0 & 0 \\
\ce_1 & t & \ce_2 & 0 & 0 & 0 & \cdots &0  & 0 & 0 \\
0 & \ce_2 & t & \ce_3 & 0  & 0 & \cdots &0  & 0 & 0 \\
0 & 0 & \ce_3 & t & \ce_4 & 0 & \cdots &0  & 0 & 0 \\
0 & 0 & 0 & \ce_4 & t & \ce_5 & \cdots &0  & 0 & 0 \\
0 & 0 & 0 & 0 & \ce_5 & t & \cdots &0  & 0 & 0 \\
\vdots & \vdots & \vdots & \vdots & \vdots & \vdots & \ddots & \vdots & \vdots & \vdots \\
0 & 0 & 0 & 0 & 0 & 0  & \cdots & t & \ce_{n-1} & 0 \\
0 & 0 & 0 & 0 & 0 & 0  & \cdots & \ce_{n-1 }& t & \ce_{n} \\
0 & 0 & 0 & 0 & 0 & 0  & \cdots & 0 & \ce_{n} & t + \ce_0 
\end{pmatrix} 
\det\begin{pmatrix}
t - \ce_0 & \ce_1 & 0 & 0 & 0 & 0 & \cdots &0  & 0 & 0 \\
\ce_1 & t & \ce_2 & 0 & 0 & 0 & \cdots &0  & 0 & 0 \\
0 & \ce_2 & t & \ce_3 & 0  & 0 & \cdots &0  & 0 & 0 \\
0 & 0 & \ce_3 & t & \ce_4 & 0 & \cdots &0  & 0 & 0 \\
0 & 0 & 0 & \ce_4 & t & \ce_5 & \cdots &0  & 0 & 0 \\
0 & 0 & 0 & 0 & \ce_5 & t & \cdots &0  & 0 & 0 \\
\vdots & \vdots & \vdots & \vdots & \vdots & \vdots & \ddots & \vdots & \vdots & \vdots \\
0 & 0 & 0 & 0 & 0 & 0  & \cdots & t & \ce_{n-1} & 0 \\
0 & 0 & 0 & 0 & 0 & 0  & \cdots & \ce_{n-1 }& t & \ce_{n} \\
0 & 0 & 0 & 0 & 0 & 0  & \cdots & 0 & \ce_{n} & t -\ce_0 
\end{pmatrix}.
\end{equation*}

{\bf Case $n+1$ is odd.}  By Lemma \ref{le:Apm}, since $r_j=r_{n-j+1}$ for $j=1, 2, 3, \dots, \frac{n}{2}$, the product \eqref{eq:FA} equals
\begin{equation}\label{eq:n+1odd}
\begin{split}
\det\begin{pmatrix}
t+\ce_0 & \ce_1 & 0 & \dots & 0 & 0 & 0 \\
\ce_1 & t & \ce_2 & \dots & 0 & 0 & 0 &  \\
0 & \ce_2 & t & \dots  & 0  & 0 & 0  \\
\vdots & \vdots & \vdots & \ddots & \vdots & \vdots & \vdots \\
0 & 0 & 0 & \dots & t & \ce_{\frac{n}{2}-1} & 0 & \\
0 & 0 & 0 & \dots & \ce_{\frac{n}{2}-1} & t  &   \ce_{\frac{n}{2}} \\
0 & 0 & 0 & \dots & 0 & 2 \ce_{\frac{n}{2}} & t \\
\end{pmatrix}
\det\begin{pmatrix}
t-\ce_0 & \ce_{1} & 0 & \cdots & 0 & 0 & 0 \\
\ce_{1} & t & \ce_{2} & \cdots & 0 & 0 & 0 \\
0 &  \ce_{2} & t & \cdots  & 0 & 0 & 0 \\
\vdots & \vdots & \ddots & \vdots & \vdots & \vdots \\
0 & 0 & 0 & \cdots & t & \ce_{\frac{n}{2}-2} & 0 \\
0 & 0 & 0 & \cdots & \ce_{\frac{n}{2}-2}  & t & \ce_{\frac{n}{2}-1} \\
0 & 0 & 0 & \cdots & 0 & \ce_{\frac{n}{2}-1} & t 
\end{pmatrix}
\\
\det\begin{pmatrix}
t-\ce_0 & \ce_1 & 0 & \dots & 0 & 0 & 0 \\
\ce_1 & t & \ce_2 & \dots & 0 & 0 & 0 &  \\
0 & \ce_2 & t & \dots  & 0  & 0 & 0  \\
\vdots & \vdots & \vdots & \ddots & \vdots & \vdots & \vdots \\
0 & 0 & 0 & \dots & t & \ce_{\frac{n}{2}-1} & 0 & \\
0 & 0 & 0 & \dots & \ce_{\frac{n}{2}-1} & t  &   \ce_{\frac{n}{2}} \\
0 & 0 & 0 & \dots & 0 & 2 \ce_{\frac{n}{2}} & t \\
\end{pmatrix}
\det\begin{pmatrix}
t+\ce_0 & \ce_{1} & 0 & \cdots & 0 & 0 & 0 \\
\ce_{1} & t & \ce_{2} & \cdots & 0 & 0 & 0 \\
0 &  \ce_{2} & t & \cdots  & 0 & 0 & 0 \\
\vdots & \vdots & \ddots & \vdots & \vdots & \vdots \\
0 & 0 & 0 & \cdots & t & \ce_{\frac{n}{2}-2} & 0 \\
0 & 0 & 0 & \cdots & \ce_{\frac{n}{2}-2}  & t & \ce_{\frac{n}{2}-1} \\
0 & 0 & 0 & \cdots & 0 & \ce_{\frac{n}{2}-1} & t 
\end{pmatrix}.
\end{split}
\end{equation}

By applying Lemma \ref{le:Bpm} to the above equality, we conclude that $P(t,x,y)$ equals
\[
\det\begin{pmatrix}
t +\ce_0 & \ce_1 & 0 & 0 &  \cdots &0  & 0 & 0 \\
\ce_1 & t & \ce_2 & 0 &  \cdots &0  & 0 & 0 \\
0 & \ce_2 & t & \ce_3 &\cdots &0  & 0 & 0 \\
0 & 0 & \ce_3 & t & \cdots &0  & 0 & 0 \\
\vdots & \vdots & \vdots & \vdots & \ddots & \vdots & \vdots & \vdots \\
0 & 0 & 0 & 0  & \cdots & t & \ce_{n-1} & 0 \\
0 & 0 & 0 & 0  & \cdots & \ce_{n-1 }& t & \ce_{n} \\
0 & 0 & 0 & 0  & \cdots & 0 & \ce_{n} & t +  \ce_0 
\end{pmatrix}
\det\begin{pmatrix}
t - \ce_0 & \ce_1 & 0 & 0 & \cdots &0  & 0 & 0 \\
\ce_1 & t & \ce_2 & 0 & \cdots &0  & 0 & 0 \\
0 & \ce_2 & t & \ce_3 & \cdots &0  & 0 & 0 \\
0 & 0 & \ce_3 & t & \cdots &0  & 0 & 0 \\
\vdots & \vdots & \vdots & \vdots & \ddots & \vdots & \vdots & \vdots \\
0 & 0 & 0 & 0  & \cdots & t & \ce_{n-1} & 0 \\
0 & 0 & 0 & 0  & \cdots & \ce_{n-1 }& t & \ce_{n} \\
0 & 0 & 0 & 0  & \cdots & 0 & \ce_{n} & t - \ce_0 \end{pmatrix}. 
\]

Lastly, if $n+1=2$, then it is straightforward to check that 
\[
P(t,x,y)=\det \begin{pmatrix} t+\ce_0 & \ce_1 \\
\ce_1 & t+\ce_0
\end{pmatrix} 
\det \begin{pmatrix} t-\ce_0 & \ce_1 \\
\ce_1 & t-\ce_0
\end{pmatrix}.
\]
Therefore, for all $n \in \N$ the polynomial $P(t,x,y)$ has the same factorization.

Now, a straightforward calculation shows that
\[
\Re(A^\pm)=
\begin{pmatrix}
\pm a_1 & \alpha_1 & 0 & 0 & \cdots & 0 & 0 & 0\\
\alpha_1 & 0 & \alpha_2 & 0 & \cdots & 0 & 0 & 0\\
0 & \alpha_2 & 0 & \alpha_3 & \cdots & 0 & 0 & 0\\
0 & 0 & \alpha_3 & 0 & \cdots & 0 & 0 & 0\\
\vdots & \vdots & \vdots & \vdots &  \ddots &\vdots & \vdots & \vdots \\
0 & 0 & 0 & 0 & \cdots & 0 & \alpha_{n-1} & 0 \\
0 & 0 & 0 & 0 & \cdots & \alpha_{n-1} & 0 & \alpha_n \\
0 & 0 & 0 & 0 & \cdots & 0 & \alpha_n & \pm c_0
\end{pmatrix}
\]
and
\[
\Im(A^\pm)=
\begin{pmatrix}
0 & \gamma_1 & 0 & 0 & \cdots & 0 & 0 & 0\\
-\gamma_1 & 0 & \gamma_2 & 0 & \cdots & 0 & 0 & 0\\
0 & -\gamma_2 & 0 & \gamma_3 & \cdots & 0 & 0 & 0\\
0 & 0 & -\gamma_3 & 0 & \cdots & 0 & 0 & 0\\
\vdots & \vdots & \vdots & \vdots &  \ddots &\vdots & \vdots & \vdots \\
0 & 0 & 0 & 0 & \cdots & 0 & \gamma_{n-1} & 0 \\
0 & 0 & 0 & 0 & \cdots & -\gamma_{n-1} & 0 & -\gamma_n \\
0 & 0 & 0 & 0 & \cdots & 0 & \gamma_n & 0
\end{pmatrix}.
\]
It follows that  $
F_{A\sp\pm}(t, x, y)={\rm det}(t I_{n+1} +x \Re(A\sp\pm) +y \Im(A\sp\pm))$ equals
\[
\det\begin{pmatrix}
t \pm\ce_0 & \lambda_{2,3} & 0 & \cdots &0  & 0 & 0 \\
\lambda_{3,2} & t & \lambda_{3,4} & \cdots &0  & 0 & 0 \\
0 & \lambda_{4,3} & t & \cdots &0  & 0 & 0 \\
\vdots & \vdots & \vdots & \ddots & \vdots & \vdots & \vdots \\
0 & 0 & 0  & \cdots & t & \lambda_{n,n+1} & 0 \\
0 & 0 & 0  & \cdots & \lambda_{n+1,n }& t & \lambda_{n+1,n+2}  \\
0 & 0 & 0  & \cdots & 0 & \lambda_{n+2,n+1} & t \pm  \ce_0 
\end{pmatrix}.
\]
Hence,
\[
P(t,x,y)=F_{A^+}(t,x,y) \cdot F_{A^-}(t,x,y)=F_A(t,x,y),
\]
as desired.
\end{proof}

The same type of conclusion given in the above theorem can be deduced if, instead of assuming $c_0=a_1$ (and the rest of the equalities) one assumes, for example, that $c_1=a_2$ (and other, similar, equalities). 

We illustrate the use of the above theorem with a couple of examples.  

\begin{example}
Let $a$ be the periodic sequence with period word $0100\cdots0$ of length $n+1$ and let $c$ be the periodic sequence  with period word $1111\cdots1$ of length $n+1$. Then, we have $1=c_0=a_1$, $1=|c_1+a_2|=|c_n+a_0|$, $1=|c_1-a_2|=|c_n-a_0|$,  $1=|c_j+a_{j+1}|=|c_{n-j+1}+a_{n-j+2}|$ and $1=|c_j-a_{j+1}|=|c_{n-j+1}-a_{n-j+2}|$  for all $j=2, 3, \dots \left[\frac{n}{2}\right]$. Therefore, the closure of the numerical range of $T(a,0,c)$ is the convex hull of the numerical ranges of the $(n+1)\times(n+1)$ matrices
\[
A^{\pm}=\begin{pmatrix}
\pm 1 & 1 & 0 & 0 & \cdots & 0 & 0 & 0\\
0 & 0 & 1 & 0 & \cdots & 0 & 0 & 0\\
0 & 0 & 0 & 1 & \cdots & 0 & 0 & 0\\
0 & 0 & 0 & 0 & \cdots & 0 & 0 & 0\\
\vdots & \vdots & \vdots & \vdots &  \ddots &\vdots & \vdots & \vdots \\
0 & 0 & 0 & 0 & \cdots & 0 & 1 & 0 \\
0 & 0 & 0 & 0 & \cdots & 0 & 0 & 1 \\
0 & 0 & 0 & 0 & \cdots & 0 & 0 & \pm 1
\end{pmatrix}.
\]
Let $a'$ be the periodic sequence with period word $000\cdots01$ of length $n+1$. By Proposition 1.1 in \cite{IM}, we have
\[
\ovl{W(T(a',0,c))}=\ovl{W(A_{a'})},
\]
where $A_{a'}$ is the biinfinite matrix in \eqref{eq:Ab}. But since the matrices $A_{a'}$ and $A_{a}$ are unitarily equivalent, then $W(A_{a'})=W(A_{a})$ (see, for example, \cite[Theorem 3.3]{HI-O2016}). Lastly, again using Proposition 1.1 in \cite{IM}, we have $\ovl{W(A_{a})}=\ovl{W(T(a,0,c))}$, which proves Conjecture 3.7 in \cite{IM}.
\end{example}

The previous example can be generalized as follows.

\begin{example}
Let $T$ be the operator with tridiagonal matrix
\[
\begin{pmatrix}
0 & 1 & 0 &   \\
a_1 & 0 & 1 & 0  \\
0  & a_{2} & 0 & 1 & \ddots \\
 & 0  & a_{3} & 0 & \ddots  & \ddots \\
 & & \ddots & \ddots & \ddots  & \ddots & \ddots \\
 & & & 0 & a_{n-1} & 0 & 1 & 0\\
 & & & & 0 & a_{n} & 0 & 1 & 0\\
 & & & & & 0 & a_{0} & 0 & 1 & 0\\
 & & & & & & 0 & a_1 & 0 & 1 & 0  \\
 &  & & & & & & 0 & a_{2} & 0 & 1 & 0  \\
 & & &  & & & & & & \ddots & \ddots & \ddots & \ddots & & \\
\end{pmatrix}. 
\]
If $a_1=1$, $a_2=a_0$, $a_3=a_n$, $a_4=a_{n-1}$, etc., then Theorem~\ref{main} applies. That is, the closure of the numerical range of $T$ is the convex hull of the numerical ranges of the matrices
\[
A^{\pm}=\begin{pmatrix}
\pm 1 & 1 & 0 & 0 & \cdots & 0 & 0 & 0\\
a_2 & 0 & 1 & 0 & \cdots & 0 & 0 & 0\\
0 & a_3 & 0 & 1 & \cdots & 0 & 0 & 0\\
0 & 0 & a_4 & 0 & \cdots & 0 & 0 & 0\\
\vdots & \vdots & \vdots & \vdots &  \ddots &\vdots & \vdots & \vdots \\
0 & 0 & 0 & 0 & \cdots & 0 & 1 & 0 \\
0 & 0 & 0 & 0 & \cdots & a_n & 0 & 1 \\
0 & 0 & 0 & 0 & \cdots & 0 & a_0 & \pm 1
\end{pmatrix}.
\]
In particular, when the periodic word is of length 2 with $a_0=-1$ and $a_1=1$ then the operator $T$ is a $2$-periodic ``hopping sign'' operator. On the other hand, $T$ is also an operator of the form $A(\infty,-1)$ as defined in \cite{CN2011}. We have that the closure of the numerical range of this $T$ is the convex hull of the numerical ranges of the matrices $A\sp{\pm}=\begin{pmatrix} \pm 1 & 1\\ -1 & \pm 1\end{pmatrix}$, that is, it is the convex hull of the sets $W(A^+)=\{1+yi\colon -1\leq y\leq 1\}$ and $W(A^-)=\{-1+yi\colon -1\leq y\leq 1\}$, thus recovering Theorem~9 in \cite{CN2011}.
\end{example}

\section{Further simplification, if $n+1$ is odd}

For the rest of this section, we assume that $n+1$ is odd. In Theorem~\ref{main} we have written the closure of the numerical range of $T$ as the convex hull of the union of the numerical ranges of two matrices of size $(n+1)\times (n+1)$. Here we simplify this result replacing the matrices by matrices of sizes $(\frac{n}{2}+1)\times(\frac{n}{2}+1)$. Assume that we have obtained the factorization we did in the proof of Theorem~\ref{main}, specifically, we have obtained Equation \eqref{eq:n+1odd}:
\begin{equation*}
\begin{split}
    P(t,x,y)&=
    \det\begin{pmatrix}
t+\ce_0 & \ce_1 & 0 & \dots & 0 & 0 & 0 \\
\ce_1 & t & \ce_2 & \dots & 0 & 0 & 0 &  \\
0 & \ce_2 & t & \dots  & 0  & 0 & 0  \\
\vdots & \vdots & \vdots & \ddots & \vdots & \vdots & \vdots \\
0 & 0 & 0 & \dots & t & \ce_{\frac{n}{2}-1} & 0 & \\
0 & 0 & 0 & \dots & \ce_{\frac{n}{2}-1} & t  &   \ce_{\frac{n}{2}} \\
0 & 0 & 0 & \dots & 0 & 2 \ce_{\frac{n}{2}} & t \\
\end{pmatrix}
\det\begin{pmatrix}
t-\ce_0 & \ce_{1} & 0 & \cdots & 0 & 0 & 0 \\
\ce_{1} & t & \ce_{2} & \cdots & 0 & 0 & 0 \\
0 &  \ce_{2} & t & \cdots  & 0 & 0 & 0 \\
\vdots & \vdots & \ddots & \vdots & \vdots & \vdots \\
0 & 0 & 0 & \cdots & t & \ce_{\frac{n}{2}-2} & 0 \\
0 & 0 & 0 & \cdots & \ce_{\frac{n}{2}-2}  & t & \ce_{\frac{n}{2}-1} \\
0 & 0 & 0 & \cdots & 0 & \ce_{\frac{n}{2}-1} & t 
\end{pmatrix}\\
&\quad
\det\begin{pmatrix}
t-\ce_0 & \ce_1 & 0 & \dots & 0 & 0 & 0 \\
\ce_1 & t & \ce_2 & \dots & 0 & 0 & 0 &  \\
0 & \ce_2 & t & \dots  & 0  & 0 & 0  \\
\vdots & \vdots & \vdots & \ddots & \vdots & \vdots & \vdots \\
0 & 0 & 0 & \dots & t & \ce_{\frac{n}{2}-1} & 0 & \\
0 & 0 & 0 & \dots & \ce_{\frac{n}{2}-1} & t  &   \ce_{\frac{n}{2}} \\
0 & 0 & 0 & \dots & 0 & {2} \ce_{\frac{n}{2}} & t \\
\end{pmatrix}
\det\begin{pmatrix}
t+\ce_0 & \ce_{1} & 0 & \cdots & 0 & 0 & 0 \\
\ce_{1} & t & \ce_{2} & \cdots & 0 & 0 & 0 \\
0 &  \ce_{2} & t & \cdots  & 0 & 0 & 0 \\
\vdots & \vdots & \ddots & \vdots & \vdots & \vdots \\
0 & 0 & 0 & \cdots & t & \ce_{\frac{n}{2}-2} & 0 \\
0 & 0 & 0 & \cdots & \ce_{\frac{n}{2}-2}  & t & \ce_{\frac{n}{2}-1} \\
0 & 0 & 0 & \cdots & 0 & \ce_{\frac{n}{2}-1} & t 
\end{pmatrix}.
\end{split}
\end{equation*}

Substituting back the values of $\lambda_{i,j}$, given in the proof of Theorem~\ref{main}, it now follows that $P(t,x,y)$ equals the product
\begin{equation*}
\begin{split}
& \det\begin{pmatrix}
t+\ce_0 & \lambda_{2,3} & 0 & \dots & 0 & 0 & 0 \\
\lambda_{3,2} & t & \lambda_{3,4} & \dots & 0 & 0 & 0 &  \\
0 & \lambda_{4,3} & t & \dots  & 0  & 0 & 0  \\
\vdots & \vdots & \vdots & \ddots & \vdots & \vdots & \vdots \\
0 & 0 & 0 & \dots & t & \lambda_{\frac{n}{2},\frac{n}{2}+1} & 0 & \\
0 & 0 & 0 & \dots & \lambda_{\frac{n}{2}+1,\frac{n}{2}} & t & \sqrt{2} \lambda_{\frac{n}{2}+1,\frac{n}{2}+2} \\
0 & 0 & 0 & \dots & 0 & \sqrt{2} \lambda_{\frac{n}{2}+2,\frac{n}{2}+1} & t \\
\end{pmatrix}
\det\begin{pmatrix}
t-\ce_0 & \lambda_{2,3} & 0 & \dots & 0 & 0 & 0 \\
\lambda_{3,2} & t & \lambda_{3,4} & \dots & 0 & 0 & 0 &  \\
0 & \lambda_{4,3} & t & \dots  & 0  & 0 & 0  \\
\vdots & \vdots & \vdots & \ddots & \vdots & \vdots & \vdots \\
0 & 0 & 0 & \dots & t & \lambda_{\frac{n}{2}-1,\frac{n}{2}} & 0 & \\
0 & 0 & 0 & \dots & \lambda_{\frac{n}{2},\frac{n}{2}-1} & t & \lambda_{\frac{n}{2},\frac{n}{2}+1} \\
0 & 0 & 0 & \dots & 0 & \lambda_{\frac{n}{2}+1,\frac{n}{2}} & t
\end{pmatrix}\\
\\
& \det\begin{pmatrix}
t-\ce_0 & \lambda_{2,3} & 0 & \dots & 0 & 0 & 0 \\
\lambda_{3,2} & t & \lambda_{3,4} & \dots & 0 & 0 & 0 &  \\
0 & \lambda_{4,3} & t & \dots  & 0  & 0 & 0  \\
\vdots & \vdots & \vdots & \ddots & \vdots & \vdots & \vdots \\
0 & 0 & 0 & \dots & t & \lambda_{\frac{n}{2},\frac{n}{2}+1} & 0 & \\
0 & 0 & 0 & \dots & \lambda_{\frac{n}{2}+1,\frac{n}{2}} & t & \sqrt{2} \lambda_{\frac{n}{2}+1,\frac{n}{2}+2} \\
0 & 0 & 0 & \dots & 0 & \sqrt{2} \lambda_{\frac{n}{2}+2,\frac{n}{2}+1} & t \\
\end{pmatrix}
\det\begin{pmatrix}
t+\ce_0 & \lambda_{2,3} & 0 & \dots & 0 & 0 & 0 \\
\lambda_{3,2} & t & \lambda_{3,4} & \dots & 0 & 0 & 0 &  \\
0 & \lambda_{4,3} & t & \dots  & 0  & 0 & 0  \\
\vdots & \vdots & \vdots & \ddots & \vdots & \vdots & \vdots \\
0 & 0 & 0 & \dots & t & \lambda_{\frac{n}{2}-1,\frac{n}{2}} & 0 & \\
0 & 0 & 0 & \dots & \lambda_{\frac{n}{2},\frac{n}{2}-1} & t & \lambda_{\frac{n}{2},\frac{n}{2}+1} \\
0 & 0 & 0 & \dots & 0 & \lambda_{\frac{n}{2}+1,\frac{n}{2}} & t \\
\end{pmatrix}.
\end{split}
\end{equation*}

Now, consider the $(\frac{n}{2}+1) \times (\frac{n}{2}+1)$ matrices
\[
B^{\pm}= \begin{pmatrix}
\pm a_1 & c_1 & 0 & 0 & \cdots & 0 & 0 & 0\\
a_2 & 0 & c_2 & 0 & \cdots & 0 & 0 & 0\\
0 & a_3 & 0 & c_3 & \cdots & 0 & 0 & 0\\
0 & 0 & a_4 & 0 & \cdots & 0 & 0 & 0\\
\vdots & \vdots & \vdots & \vdots &  \ddots &\vdots & \vdots & \vdots \\
0 & 0 & 0 & 0 & \cdots & 0 & c_{\frac{n}{2}-1} & 0 \\
0 & 0 & 0 & 0 & \cdots & a_{\frac{n}{2}} & 0 & \sqrt{2} c_{\frac{n}{2}} \\
0 & 0 & 0 & 0 & \cdots & 0 & \sqrt{2} a_{\frac{n}{2}+1} & 0
\end{pmatrix}
\]
and its $\frac{n}{2} \times \frac{n}{2}$ submatrices
\[
B_1^{\pm}= \begin{pmatrix}
\pm a_1 & c_1 & 0 & 0 & \cdots & 0 & 0 \\
a_2 & 0 & c_2 & 0 & \cdots & 0 & 0 \\
0 & a_3 & 0 & c_3 & \cdots & 0 & 0 \\
0 & 0 & a_4 & 0 & \cdots & 0 & 0 \\
\vdots & \vdots & \vdots & \vdots &  \ddots &\vdots & \vdots\\
0 & 0 & 0 & 0 & \cdots & 0 & c_{\frac{n}{2}-1} \\
0 & 0 & 0 & 0 & \cdots & a_{\frac{n}{2}} & 0
\end{pmatrix}
\]
Given the factorization above, it is then straightforward to check that
\[
P(t,x,y)=F_{B^+}(t,x,y) \cdot F_{B_1^-}(t,x,y) \cdot F_{B^-}(t,x,y) \cdot F_{B_1^+}(t,x,y),
\]
and hence, the numerical range of the tridiagonal operator $T$ is the convex hull of the numerical ranges of the four matrices $B^+$, $B^-$, $B_1^+$, and $B_1^-$. But, since the matrices $B_1^\pm$ are submatrices of $B^\pm$, the closure of the numerical range of $T$  is the convex hull of the numerical ranges of the two matrices $B^\pm$, which are of size $(\frac{n}{2}+1)\times (\frac{n}{2}+1)$. We have proved the following corollary.

\begin{corollary}\label{cor_main}
For $n \in \N$, with $n+1$ odd, let 
$a$ and  $c$  be $n+1$-periodic infinite sequences in $\mathcal \R^{\N_0}$. Let
$a_0 a_1 a_2 \cdots a_n$  and $c_0 c_1 c_2 \cdots c_n$ be the period words of $a$ and $c$, respectively.
Suppose that 
\[
c_0=a_1, \quad |c_1+a_2| =|c_n+a_0|, \quad |c_1-a_2|=|c_n-a_0|,
\]
\begin{equation*}
|c_j+a_{j+1}| = |c_{n-j+1}+a_{n-j+2}|, \quad \text{ and } \quad |c_j-a_{j+1}|  =|c_{n-j+1}-a_{n-j+2}|,
\end{equation*} 
for all $j=2, 3, \dots \frac{n}{2}$. Then 
\[
    \ovl{W(T(a, 0,c))}=\conv{W(B\sp{+})\cup W(B\sp{-}}),   \]
    where $B^{\pm}$ are the $(\frac{n}{2}+1) \times (\frac{n}{2}+1)$ matrices
\[
B^{\pm}= \begin{pmatrix}
\pm a_1 & c_1 & 0 & 0 & \cdots & 0 & 0 & 0\\
a_2 & 0 & c_2 & 0 & \cdots & 0 & 0 & 0\\
0 & a_3 & 0 & c_3 & \cdots & 0 & 0 & 0\\
0 & 0 & a_4 & 0 & \cdots & 0 & 0 & 0\\
\vdots & \vdots & \vdots & \vdots &  \ddots &\vdots & \vdots & \vdots \\
0 & 0 & 0 & 0 & \cdots & 0 & c_{\frac{n}{2}-1} & 0 \\
0 & 0 & 0 & 0 & \cdots & a_{\frac{n}{2}} & 0 & \sqrt{2} c_{\frac{n}{2}} \\
0 & 0 & 0 & 0 & \cdots & 0 & \sqrt{2} a_{\frac{n}{2}+1} & 0
\end{pmatrix}.
\]
\end{corollary}

We illustrate this with an example.

\begin{example}
Let $a$ be the periodic sequence $(0,1,0,0, \dots, 0)$ of length $n+1$ odd and let $c$ be the periodic sequence $(1,1,1,1, \dots, 1)$ of length $n+1$. Then, the closure of the numerical range of $T(a,0,c)$ is the convex hull of the numerical ranges of the $(\frac{n}{2}+1) \times (\frac{n}{2}+1)$ matrices
\[
B^{\pm}=\begin{pmatrix}
\pm 1 & 1 & 0 & 0 & \cdots & 0 & 0 & 0\\
0 & 0 & 1 & 0 & \cdots & 0 & 0 & 0\\
0 & 0 & 0 & 1 & \cdots & 0 & 0 & 0\\
0 & 0 & 0 & 0 & \cdots & 0 & 0 & 0\\
\vdots & \vdots & \vdots & \vdots &  \ddots &\vdots & \vdots & \vdots \\
0 & 0 & 0 & 0 & \cdots & 0 & 1 & 0 \\
0 & 0 & 0 & 0 & \cdots & 0 & 0 & \sqrt{2} \\
0 & 0 & 0 & 0 & \cdots & 0 & 0 & 0
\end{pmatrix}.
\]
Observe that if $n=2$, then $B^{\pm}=\begin{pmatrix} \pm 1 & \sqrt{2} \\ 0 & 0 \end{pmatrix}$. This is the case $n+1=3$ in Conjecture 3.7 in \cite{IM} and illustrates the reason why the numerical range of $W(T(a,0,c))$ is the convex hull of two ellipses, as shown in Figure 2 in \cite{IM}.
\end{example}

\bibliographystyle{plain}

\end{document}